\documentclass[11pt]{article}

\usepackage[T1]{fontenc}
\usepackage{lmodern}
\usepackage{microtype}
\usepackage[margin=1in]{geometry}
\usepackage{amsmath,amssymb,amsthm}
\usepackage[hidelinks]{hyperref}

\hypersetup{
  pdftitle={Neumaier graphs of coherent rank five},
  pdfauthor={Gary Greaves and Zhao Kuang Tan}
}

\newcommand{\F}{\mathbb F}
\newcommand{\one}{\mathbf 1}
\newcommand{\cA}{\mathcal A}
\newcommand{\cB}{\mathcal B}
\newcommand{\cW}{\mathcal W}
\newcommand{\Adj}{A_{\Gamma}}
\newcommand{\bL}{\mathbf L}

\newtheorem{theorem}{Theorem}[section]
\newtheorem{proposition}[theorem]{Proposition}

\theoremstyle{definition}

\newtheorem{example}[theorem]{Example}
\theoremstyle{remark}

\newtheorem{question}[theorem]{Question}
\numberwithin{equation}{section}

\makeatletter
\newcommand*{\transpose}{%
  {\mathpalette\@transpose{}}%
}
\newcommand*{\@transpose}[2]{%
  \raisebox{\depth}{$\m@th#1\intercal$}%
}
\makeatother

\title{Neumaier graphs of coherent rank five}
\author{
  Gary Greaves\thanks{School of Physical and Mathematical Sciences,
  Nanyang Technological University, 21 Nanyang Link, Singapore 637371;
  \texttt{gary@ntu.edu.sg}}
  \and
  Zhao Kuang Tan\thanks{School of Physical and Mathematical Sciences,
  Nanyang Technological University, 21 Nanyang Link, Singapore 637371;
  \texttt{zhaokuan001@e.ntu.edu.sg}}
}
\date{}

\begin{document}
\maketitle

\begin{abstract}
We construct an infinite family of Neumaier graphs of coherent rank five, answering the existence question at the smallest possible coherent rank beyond the strongly regular case. For every prime power $q\geqslant7$ with $q\equiv3\pmod4$, set $n=q+1$. 
Each graph in our construction has precisely five distinct eigenvalues, Neumaier parameters
\[
 \left(
 n(n-1)(n-3),
 \frac{n^2(n-3)}2,
 \frac{n(n^2-n-8)}4;
 \frac{(n-2)^2}{2},
 (n-1)(n-3)
 \right),
\]
and its adjacency matrix lies in the Bose--Mesner algebra of the four-class association scheme of Holzmann, Kharaghani, and Suda. 
Paley Hadamard matrices and Desarguesian mutually orthogonal Latin squares yield an infinite family whose smallest member has parameters 
$(280,160,96;18,35)$.
\end{abstract}



\section{Introduction}

A finite simple graph is \textbf{$(v,k,\lambda)$-edge-regular} if it has $v$ vertices, is $k$-regular, and every adjacent pair has exactly $\lambda$ common neighbours. 
A proper non-empty clique $C$ of order $s$ is \textbf{$e$-regular} if every vertex outside $C$ is adjacent to precisely $e>0$ vertices in $C$.
A non-complete edge-regular graph containing an $e$-regular clique is a \textbf{Neumaier graph}; it is \textbf{strictly Neumaier} if it is not strongly regular \cite{neumaier1981}. 
We write $(v,k,\lambda;e,s)$ for its parameters.

Neumaier asked whether every Neumaier graph is strongly regular \cite[p.~248]{neumaier1981}. Greaves and Koolen answered this question negatively by constructing an infinite family of strictly Neumaier graphs \cite{greavesKoolen2018}, followed by a second construction in \cite{greavesKoolen2018b}. Further constructions have arisen from affine polar graphs \cite{evansGoryainovPanasenko2019}, switching and a general nexus-one construction \cite{evansEtAl2023}, Jacobi sums \cite{abiadEtAl2023}, and cyclotomy \cite{greavesTan2026}.
Small-order existence and non-existence were studied further in \cite{abiadDeBoeckZeijlemaker2024}. Recent developments include Neumaier Cayley graphs \cite{evansEtAl2026} and examples with exactly five adjacency eigenvalues \cite{deBruynEtAl2026}.
The coherent rank of a graph is the dimension of its coherent closure \cite{higman1987}. 
Non-trivial strongly regular graphs have coherent rank three, whereas every Neumaier graph of coherent rank at most four is strongly regular \cite{abiadEtAlFewEigenvalues,greavesTan2026}. 
Thus, coherent rank five is the first possible rank beyond the strongly regular case, and the existence of Neumaier graphs of coherent rank five was asked in \cite[Question~2.4]{greavesTan2026}.

Our main result is the following.

\begin{theorem}\label{thm:main}
For every prime power $q\geqslant7$ with $q\equiv3\pmod4$, there exists a
Neumaier graph of coherent rank five with parameters
\[
 \left(
 (q+1)q(q-2),
 \frac{(q+1)^2(q-2)}2,
 \frac{(q+1)(q^2+q-8)}4;
 \frac{(q-1)^2}{2},
 q(q-2)
 \right).
\]
\end{theorem}

As $q$ ranges over $3^{2a+1}$ for $a\geqslant1$, the graphs of Theorem~\ref{thm:main} form an infinite family of pairwise non-isomorphic graphs.
Theorem~\ref{thm:main} provides the first known infinite family of Neumaier graphs with coherent rank five, and also, to our knowledge, the first infinite family of Neumaier graphs having five distinct eigenvalues.
The adjacency matrix of each graph of Theorem~\ref{thm:main} belongs to the Bose--Mesner algebra of a four-class association scheme constructed by Holzmann, Kharaghani, and Suda~\cite{holzmannKharaghaniSuda2015}.

Our construction requires a Hadamard matrix $H$ of order $n$ and a
family of $n-3$ mutually appropriate Latin squares of order $n-1$
(equivalently, the existence of $n-3$ mutually orthogonal Latin
squares of order $n-1$).
We provide a more general construction (see Theorem~\ref{thm:main2}) based on these ingredients in Section~4.
The well-known Paley Hadamard matrices and Desarguesian mutually orthogonal Latin squares then give rise to Theorem~\ref{thm:main}, above. 

The paper is organised as follows. 
In Section~2, we fix our conventions for coherent algebras and association schemes.
In Section~3, we introduce MALS, whose existence is equivalent to that of MOLS.
We recall the
four-class association scheme of Holzmann, Kharaghani, and Suda, and define the graph
$\Gamma(H;\bL)$, which depends on a Hadamard matrix $H$ and a family of MALS $\bL$. 
In Section~4, we show that $\Gamma(H;\bL)$ has a spread of regular cliques and provide a condition that guarantees that $\Gamma(H;\bL)$ is edge regular.
Finally, in Section~5 we outline some further questions.

\section{Coherent algebras and association schemes}

\subsection{Coherent algebras and coherent closure}

We use the coherent-algebra and association-scheme conventions \cite[Section~2]{greavesTan2026}, with coherent algebras taken over $\mathbb R$. 
For completeness, we recall the definitions and notation needed below.

Let $X$ be a finite set, and write $M\circ N$ for the entrywise product of matrices $M$ and $N$.
A \textbf{coherent algebra} on $X$ is a real vector subspace $\mathcal C\subseteq\mathbb R^{X\times X}$ such that
\begin{enumerate}
\setlength{\itemsep}{0pt}
\setlength{\parskip}{0pt}
\item $I,J\in\mathcal C$;
\item $M^\transpose\in\mathcal C$ for every $M\in\mathcal C$;
\item $MN,M\circ N\in\mathcal C$ for every $M,N\in\mathcal C$.
\end{enumerate}

Every coherent algebra has a unique standard basis $\{A_1,\dots,A_r\}$ of pairwise disjoint $\{0,1\}$-matrices whose sum is $J$; the corresponding binary relations form its underlying \textbf{coherent configuration} \cite[Section~3]{higman1987}. 
The matrices $A_1,\dots,A_r$ are called the \textbf{relation matrices} of the coherent configuration.
For a graph $\Gamma$ on $X$ with adjacency matrix $A_\Gamma$, its \textbf{adjacency algebra} is $\cA(\Gamma):=\mathbb R[A_\Gamma]=\{p(A_\Gamma):p\in\mathbb R[x]\}$. 
Since $A_\Gamma$ is a real symmetric matrix, $\dim\cA(\Gamma)$ is the number of distinct adjacency eigenvalues. 
The \textbf{coherent closure} $\cW(\Gamma)$ is the smallest coherent algebra containing $A_\Gamma$.
The \textbf{coherent rank} of $\Gamma$ is defined to be the rank of its coherent closure $\cW(\Gamma)$. 
We use the standard equivalent characterisation that $\Gamma$ is \textbf{quotient-polynomial} precisely when $\cA(\Gamma)=\cW(\Gamma)$ \cite{fiol2016,greavesTan2026}.

\subsection{Association schemes}

A coherent configuration is \textbf{homogeneous} when $I$ is one of its standard basis matrices. 
We call a homogeneous coherent configuration an \textbf{association scheme}; all schemes below are symmetric, meaning that every relation matrix is symmetric. 
For a symmetric $d$-class association scheme, write the relation matrices as $A_0=I,A_1,\dots,A_d$, and let $R_i$ denote the relation with adjacency matrix $A_i$. 
Since the standard basis spans an algebra, for each $i,j$ there are unique non-negative integers $p_{ij}^{\,0},\dots,p_{ij}^{\,d}$ such that
\[
 A_iA_j=\sum_{t=0}^d p_{ij}^{\,t}A_t.
\]
These coefficients are the \textbf{intersection numbers}: if $(x,y)\in R_t$, then $p_{ij}^{\,t}$ counts the vertices $z$ such that $(x,z)\in R_i$ and $(z,y)\in R_j$.

The \textbf{Bose--Mesner algebra} is $\cB=\operatorname{span}_{\mathbb R}\{A_0,\dots,A_d\}$, a coherent algebra of dimension $d+1$. 
It also has a basis of primitive idempotents $E_0,\dots,E_d$. 
We use the eigenmatrix convention
\[
 A_i=\sum_{j=0}^d P_{ji}E_j,
 \qquad
 E_j=\frac1{|X|}\sum_{i=0}^d Q_{ij}A_i,
\]
so $P$ and $Q$ are the first and second eigenmatrices, respectively \cite[p.~150]{holzmannKharaghaniSuda2015}. 
If $A_\Gamma$ is a sum of basis matrices, then $\cA(\Gamma)\subseteq\cW(\Gamma)\subseteq\cB$; when $\cA(\Gamma)=\cB$, we say that $A_\Gamma$ \textbf{generates} the Bose--Mesner algebra.

\section{The graph construction}

\subsection{Mutually appropriate Latin squares}

A \textbf{Latin square} of order $t$ is a $t\times t$ array in which every
symbol occurs exactly once in each row and column. Two Latin squares with the
same column and symbol sets are \textbf{appropriate} if each row of one agrees
with each row of the other in exactly one column. A pairwise appropriate
family is called a family of \textbf{mutually appropriate Latin squares} (MALS).
MALS are known as \textit{mutually suitable Latin squares} (MSLS) in
\cite{holzmannKharaghaniOrrick2010} and as \textit{unique fixed symbol} (UFS) Latin squares in
\cite{kharaghaniSuda2018}.

Two Latin squares with the same cell and symbol sets are \textbf{orthogonal} if,
when superimposed, every ordered pair of symbols occurs exactly once. A
pairwise orthogonal family is called a family of \textbf{mutually orthogonal Latin
squares} (MOLS).

MALS are equivalent to MOLS~\cite[Lemma 9]{holzmannKharaghaniOrrick2010}. 
For the reader's convenience, we repeat the argument.
A Latin square $L$ can be regarded as a set of triples $((i,j),k)$, where $L(i,j)=k$.
Denote by $L^\dagger$ the image of $L$ under the row--symbol interchange
\begin{equation}\label{eq:row-symbol-interchange}
 T\colon ((i,j),k)\mapsto((k,j),i).
\end{equation}
In other words, $L^\dagger(k,j)=i$ if and only if $L(i,j)=k$.
Then $L^\dagger$ is again a Latin square. 
Indeed, each column of $L^\dagger$ is the inverse of the corresponding column permutation of $L$, while the fact that row $i$ of $L$ contains $k$ in a unique column means that row $k$ of $L^\dagger$ contains $i$ exactly once. 
For fixed symbols \(a,b\), an agreement
$L^\dagger(a,c)=M^\dagger(b,c)=r$
is equivalent to the cell \((r,c)\) containing the ordered pair \((a,b)\) in the superposition of \(L\) and \(M\). 
Thus, uniqueness of the agreement column is precisely orthogonality.
The interchange in \eqref{eq:row-symbol-interchange} is an involution.

Using the standard Desarguesian construction of MOLS~\cite[Construction~3.29]{HCD}, we apply \eqref{eq:row-symbol-interchange} to obtain a corresponding construction of MALS.

\begin{proposition}
\label{prop:desarguesian-mals}
Let $q$ be a prime power. 
For each $\alpha\in\F_q \backslash \{0\}$, define the Latin square $L_\alpha(r,c)=\alpha(r-c)$, where $r,c \in \mathbb F_q$.
Then $\{L_\alpha \;:\; \alpha\in\F_q \backslash \{0\} \}$ is a family of $q-1$ MALS of order $q$.
\end{proposition}

Throughout, a family of MALS of order $n-1$ has row and column labels
$\{1,\dots,n-1\}$ and symbol set $\{2,\dots,n\}$.

\subsection{The four-class association scheme of Holzmann, Kharaghani, and Suda}

We follow \cite[Section~5.3]{kharaghaniSuda2018} to define a four-class association scheme originally defined in \cite[Theorem~3.1]{holzmannKharaghaniSuda2015}.

A \textbf{Hadamard matrix} of order $n$ is a matrix
$H\in\{\pm1\}^{n\times n}$ satisfying $HH^\transpose=nI_n$. After negating
columns if necessary, assume that its first row is $\one^\transpose$. Write
the rows of $H$ as
\[
 h_1=\one^\transpose,h_2,\dots,h_n.
\]
Let $\bL=\{L_1,\dots,L_m\}$ be a family of MALS of order $n-1$, where $m\geqslant3$. 
Thus the row and column labels are $\{1,\dots,n-1\}$ and the symbol set is $\{2,\dots,n\}$.

For distinct $i,j$ and row labels $r,s$, let $L_{ij}(r,s)$ be the common
symbol in the unique column $c$ satisfying
\[
 L_i(r,c)=L_j(s,c).
\]
As shown in \cite[Section~2.3]{kharaghaniSuda2018}, $L_{ij}$ is itself a Latin square.

For $\ell\in\{2,\dots,n\}$, define $ \mathfrak H_\ell:=h_\ell^\transpose h_\ell$.
Following the proof of \cite[Theorem~2.1]{holzmannKharaghaniSuda2015}, replace each symbol
$\ell$ in $L_i$ by $\mathfrak H_\ell$, obtaining a block matrix $M_i$ of order $n(n-1)$. 
By \cite[Lemma~2.1]{holzmannKharaghaniSuda2015}, for $2\leqslant a,b\leqslant n$, the matrices $\mathfrak H_\ell$ satisfy $\mathfrak H_a\mathfrak H_b=n\delta_{ab}\mathfrak H_a$.
Therefore, for $i\neq j$,
the $(r,s)$-block of $M_iM_j^\transpose$ is
\[
 \begin{aligned}
 \left [M_iM_j^\transpose\right]_{r,s}
 &=\sum_{c=1}^{n-1}
   \mathfrak H_{L_i(r,c)}\mathfrak H_{L_j(s,c)}\\
 &=n\mathfrak H_{L_{ij}(r,s)}.
 \end{aligned}
\]
Indeed, all terms vanish except the one corresponding to the unique column in which row $r$ of $L_i$ and row $s$ of $L_j$ agree. 



Set 
\[
 \Omega=\{(i,r,x):
  1\leqslant i\leqslant m,
  1\leqslant r\leqslant n-1,
  1\leqslant x\leqslant n\}
\]
and $\Omega_{i,r}=\{(i,r,x):1\leqslant x\leqslant n\}$.
Then the five matrices in the proof of
\cite[Theorem~3.1]{holzmannKharaghaniSuda2015} are $\{0,1\}$-matrices on the set $\Omega$ given by
\[
\begin{aligned}
 A_0&=I_m\otimes I_{n-1}\otimes I_n,\\
 A_1&=I_m\otimes I_{n-1}\otimes(J_n-I_n),\\
 A_4&=I_m\otimes(J_{n-1}-I_{n-1})\otimes J_n,
\end{aligned}
\]
and, for $i\neq j$,
\[
\begin{aligned}
 A_2[\Omega_{i,r},\Omega_{j,s}]
   &=\frac12 \left (J_n + \mathfrak H_{L_{ij}(r,s)} \right ),\\
 A_3[\Omega_{i,r},\Omega_{j,s}]
   &=\frac12 \left (J_n - \mathfrak H_{L_{ij}(r,s)}\right ).
\end{aligned}
\]
For $i=j$, all blocks of $A_2$ and $A_3$ are zero.
By the proof of \cite[Theorem~3.1]{holzmannKharaghaniSuda2015}, the matrices
$A_0,\dots,A_4$ form a symmetric four-class association scheme. 
With respect to the ordering
$A_0,A_1,A_2,A_3,A_4$, the first and second eigenmatrices of this scheme are
given in \cite[Remark~3.1]{holzmannKharaghaniSuda2015} as
\[
 P=
 \begin{bmatrix}
 1&n-1&\dfrac{n(n-1)(m-1)}2&\dfrac{n(n-1)(m-1)}2&n(n-2)\\
 1&-1&\dfrac{n(m-1)}2&-\dfrac{n(m-1)}2&0\\
 1&n-1&0&0&-n\\
 1&-1&-\dfrac n2&\dfrac n2&0\\
 1&n-1&-\dfrac{n(n-1)}2&-\dfrac{n(n-1)}2&n(n-2)
 \end{bmatrix}
\]
and
\[
 Q=
 \begin{bmatrix}
 1&(n-1)^2&(n-2)m&(n-1)^2(m-1)&m-1\\
 1&-n+1&(n-2)m&-(n-1)(m-1)&m-1\\
 1&n-1&0&-n+1&-1\\
 1&-n+1&0&n-1&-1\\
 1&0&-m&0&m-1
 \end{bmatrix}.
\]

\section{The graph, its regular cliques, and edge regularity}

Given a Hadamard matrix $H$ and a family of MALS $\bL$, we can define the graph $\Gamma=\Gamma(H;\bL)$ by its adjacency matrix $\Adj=A_2+A_4$
where $A_2$ and $A_4$ are the relation matrices of the four-class association scheme defined above.

\subsection{A spread of regular cliques}

A \textbf{spread} in a graph is a set of cliques forming a partition of the vertex set.
We show that the graph $\Gamma(H;\bL)$ has a spread of regular cliques.

\begin{proposition}\label{prop:hadamard-column-clique-spread}
Let $n\geqslant8$ and $m\geqslant3$. 
Suppose $H$ is a Hadamard matrix of order $n$ with first row $h_1=\one^\transpose$, and let $\bL$ be a family of $m$ MALS of order $n-1$. 
For each $x\in\{1,\dots,n\}$, put
\[
 \mathcal C_x=\{(i,r,x):1\leqslant i\leqslant m,
 1\leqslant r\leqslant n-1\}.
\]
Then $\mathcal C_1,\dots,\mathcal C_n$ partition the vertex set of
$\Gamma(H;\bL)$ into cliques of size $m(n-1)$, each regular with nexus
$(n-2)+(m-1)(n/2-1)$.
\end{proposition}

\begin{proof}
Let $\bL=\{L_1,\dots,L_m\}$.
The sets $\mathcal C_1,\dots,\mathcal C_n$ partition $\Omega$, and each has
size $m(n-1)$.
Let $u=(i,r,x)$ and $v=(j,s,x)$ be distinct vertices of $\mathcal C_x$. 
If $i=j$, then $r\neq s$.
Hence, $u$ and $v$ are adjacent by the definition of $A_4$. 
If $i\neq j$, then
\[
 \left [\frac12 (J_n + \mathfrak H_{L_{ij}(r,s)})\right ]_{xx}=1.
\]
Hence, $u$ and $v$  are adjacent by the definition of $A_2$.
Thus, $\mathcal C_x$ is a clique.

Now fix $u=(i,r,y)\notin\mathcal C_x$, so $y\neq x$. 
Define $\Omega_i=\bigcup_{r=1}^{n-1}\Omega_{i,r}$.
Then, within $\Omega_i$, by the definition of $A_4$, the vertex $u$ is adjacent to $(i,s,x)$ for every $s\neq r$, giving $n-2$ neighbours in $\mathcal C_x\cap\Omega_i$.

Fix $j\neq i$. 
As $s$ varies, since $L_{ij}$ is a Latin square, the symbols $L_{ij}(r,s)$ run once through $\{2,\dots,n\}$. 
Hence, the number of neighbours of $u$ in
$\mathcal C_x\cap\Omega_j$ is
\[
\begin{aligned}
 \sum_{\ell=2}^{n}\left [\frac12 (J_n + \mathfrak H_{\ell})\right ]_{yx}
 &=\frac12\left(
    n-1+\sum_{\ell=2}^{n}h_\ell(y)h_\ell(x)
   \right)\\
 &=\frac12(n-1-1)
  =\frac n2-1.
\end{aligned}
\]
Indeed, columns $x$ and $y$ of $H$ are orthogonal, while
$h_1(x)h_1(y)=1$. 
Summing over all $\Omega_j$ yields nexus $n-2 + (m-1)(n-2)/2$, as required.
\end{proof}

\subsection{Edge regularity}

Finally, we determine when the graph $\Gamma(H;\bL)$ is edge regular.

\begin{proposition}
\label{prop:edge-regularity-criterion}
Let $n\geqslant8$ and $m\geqslant3$. 
Suppose $H$ is a Hadamard matrix of order $n$ with first row $h_1=\one^\transpose$, and let $\bL$ be a family of $m$ MALS of order $n-1$. 
Let $A$ be the adjacency matrix of $\Gamma(H;\bL)$.
Then
\[
 A^2
 =
 \left ((m-1)\frac{n(n-1)}2+n(n-2)\right )A_0+\mu_{A_1}A_1+\lambda_{A_2}A_2
 +\mu_{A_3}A_3+\lambda_{A_4}A_4,
\]
where
\begin{equation}\label{eq:common-neighbour-numbers}
\begin{aligned}
 \lambda_{A_2}
 &=(m-2)\frac{n^2}{4}+n(n-2),
 &\lambda_{A_4}
 &=(m-1)\frac{n(n-1)}4+n(n-3),\\
 \mu_{A_1}
 &=(m-1)\frac{n(n-2)}4+n(n-2),
 &\mu_{A_3}
 &=(m-2)\frac{n(n-2)}4+n(n-2).
\end{aligned}
\end{equation}
Furthermore,
$\Gamma(H;\bL)$ is edge regular if and only if $m=n-3$.
\end{proposition}

\begin{proof}
The multiplication table in the proof of
\cite[Theorem~3.1]{holzmannKharaghaniSuda2015} gives
\[
\begin{aligned}
 A_2^2&=\frac{n(n-1)(m-1)}2A_0
 +\frac{n(n-2)(m-1)}4A_1
 +\frac{n^2(m-2)}4A_2\\
 &\quad+\frac{n(n-2)(m-2)}4A_3
 +\frac{n(n-1)(m-1)}4A_4,\\
 A_2A_4&=\frac{n(n-2)}2(A_2+A_3),\\
 A_4^2&=n(n-2)(A_0+A_1)+n(n-3)A_4.
\end{aligned}
\]
Since $A^2=A_2^2+2A_2A_4+A_4^2$, collecting coefficients gives \eqref{eq:common-neighbour-numbers}. 
Since $A = A_2+A_4$, the graph $\Gamma(H;\bL)$ is edge regular if and only if $\lambda_{A_2}=\lambda_{A_4}$.
The conclusion follows from the fact that \(\lambda_{A_2}-\lambda_{A_4}
=\frac n4(m-n+3)\).
\end{proof}

\subsection{Five eigenvalues and coherent rank five}

Now we can prove our general construction theorem.

\begin{theorem}
\label{thm:main2}
    Let $n\geqslant8$. 
Suppose $H$ is a Hadamard matrix of order $n$ with first row $h_1=\one^\transpose$, and let $\bL$ be a family of $n-3$ MALS of order $n-1$. 
Then $\Gamma(H;\bL)$ is a Neumaier graph of coherent rank five with parameters
\[
 (v,k,\lambda;e,s)=
 \left(
 n(n-1)(n-3),
 \frac{n^2(n-3)}2,
 \frac{n(n^2-n-8)}4;
 \frac{(n-2)^2}{2},
 (n-1)(n-3)
 \right),
\]
and spectrum
\[
 \left\{
 \begin{gathered}
 \left[\frac{n^2(n-3)}2\right]^1,
 \left[\frac{n(n-3)}2\right]^{n-4},
 \left[\frac{n(n-4)}2\right]^{(n-1)^2},\\
 \left[-\frac n2\right]^{(n-1)^2(n-4)},
 [-n]^{(n-2)(n-3)}
 \end{gathered}
 \right\}.
\]
\end{theorem}
\begin{proof}
    Let $A$ be the adjacency matrix of the graph $\Gamma = \Gamma(H;\bL)$.
By Proposition~\ref{prop:hadamard-column-clique-spread} and Proposition~\ref{prop:edge-regularity-criterion}, the graph $\Gamma$ is a Neumaier graph.
It remains to consider its coherent rank and eigenvalues.

Specialise the first and second eigenmatrices $P$ and $Q$ to $m=n-3$. 
The sum of the $A_2$- and $A_4$-columns of $P$ gives the eigenvalues of $A$, while the first row of $Q$ gives their corresponding multiplicities.
Hence, when $m = n-3$, we have
\[
 \operatorname{Spec}(A)=
 \left\{
 \begin{gathered}
 \left[\frac{n^2(n-3)}2\right]^1,
 \left[\frac{n(n-3)}2\right]^{n-4},
 \left[\frac{n(n-4)}2\right]^{(n-1)^2},\\
 \left[-\frac n2\right]^{(n-1)^2(n-4)},
 [-n]^{(n-2)(n-3)}
 \end{gathered}
 \right\}.
\]
All five eigenvalues are distinct for $n\geqslant8$, so $\dim\cA(\Gamma)=5$. 
Let $\cB=\operatorname{span}_{\mathbb R}\{A_0,\dots,A_4\}$.
Since $A\in\cB$ and $\dim\cB=5$, it follows that $
 \cA(\Gamma)=\cW(\Gamma)=\cB$.
Thus, $A$ generates its coherent closure, that is, $\Gamma$ is quotient-polynomial of coherent rank five.
\end{proof}

Theorem~\ref{thm:main} now follows as a corollary of Theorem~\ref{thm:main2}.
Indeed, set $n=q+1$ and take $H$ to be a Paley Hadamard matrix of order $n$ \cite{paley1933}. 
By Proposition~\ref{prop:desarguesian-mals}, there is a family of $q-1=n-2$ MALS of order $q=n-1$. 
Choose any $q-2=n-3$ of them.
The result now follows from Theorem~\ref{thm:main2}.

\section{Concluding remarks}

We have provided the first known infinite family of Neumaier graphs having precisely five distinct eigenvalues and the first known examples of Neumaier graphs having coherent rank five.

\begin{example}
The smallest order example from Theorem~\ref{thm:main} has $q=7$, where $n=8$ and $m=5$. 
This graph is a Neumaier graph with parameters $(280,160,96;18,35)$, and spectrum 
\[
 \{[160]^1,[20]^4,[16]^{49},[-4]^{196},[-8]^{30}\}.
\]
\end{example}

The following questions immediately arise.

\begin{question}
What is the smallest Neumaier graph having coherent rank five?
\end{question}

An association scheme with relation matrices $A_0,A_1,\dots,A_d$ is called \textbf{primitive} if the graphs corresponding to each of the adjacency matrices $A_1, \dots, A_d$ are all connected; otherwise the association scheme is called \textbf{imprimitive}.
Moreover, a \textbf{parabolic} of an association scheme is an equivalence relation that is a union of basis relations; see, for example, \cite{higman1995}.
The Holzmann--Kharaghani--Suda four-class association scheme has a non-trivial parabolic and is therefore imprimitive. 

\begin{question}
Must the coherent closure of every Neumaier graph of coherent rank five be imprimitive? More strongly, must it contain a non-trivial parabolic whose classes are independent sets?
\end{question}

The parameters and spectrum of $\Gamma(H;\bL)$ depend only on $n$,
but its isomorphism type may depend on the choices of $H$ and $\bL$.

\begin{question}
For fixed $n$, how does the isomorphism type of $\Gamma(H;\bL)$
depend on the Hadamard matrix $H$ and the family of MALS $\bL$?
In particular, when do inequivalent choices produce isomorphic
graphs?
\end{question}

\end{document}